\documentclass[11pt]{amsart}
\usepackage{amssymb, amsthm, graphicx, graphs, epsfig, verbatim, amscd,
enumerate, stmaryrd, amsmath, psfrag}

\newtheorem{thm}{Theorem}[section]
\newtheorem{cor}[thm]{Corollary}
\newtheorem{lem}[thm]{Lemma}
\newtheorem{prop}[thm]{Proposition}

\theoremstyle{definition}
\newtheorem{defn}[thm]{Definition}

\theoremstyle{remark}
\newtheorem{rem}[thm]{Remark}

\numberwithin{equation}{section}

\newcommand{\al}{\alpha}
\newcommand{\be}{\beta}
\newcommand{\ga}{\gamma}

\newcommand{\de}{\delta}
\newcommand{\ep}{\varepsilon}

\newcommand{\ka}{\kappa}

\newcommand{\la}{\lambda}

\newcommand{\si}{\sigma}

\newcommand{\csi}{\xi}

\newcommand{\x}{\times}
\newcommand{\s}{\mathbf s}

\newcommand{\CC}{\mathcal C}
\newcommand{\f}{{\mathfrak {f}}}

\renewcommand{\SS}{{\mathfrak {F}}}
\renewcommand{\s}{{\mathfrak {s}}}

\newcommand{\ii}{{\rm II}}
\newcommand{\iii}{{\rm III}}

\newcommand{\imm}{{\mathrm {Imm}}}
\newcommand{\DIFF}{{\mathrm {DIFF}}}
\newcommand{\AUT}{{\mathrm {AUT}}}

\newcommand{\Z}{\mathbb Z}

\newcommand{\Q}{\mathbb Q}
\newcommand{\R}{\mathbb R}

\newcommand{\RP}{{\mathbb R}{P}}
\newcommand{\CP}{{\mathbb C}{P}}

\newcommand{\del}{\partial}

\newcommand{\co}{\colon\thinspace}

\newcommand{\rank}[1]{{\mathrm {rank}}\thinspace {#1}}

\newcommand{\fr}{{\mathrm {fr}}}
\newcommand{\plto}[1]
 {\xrightarrow{#1}}

\begin{document}
\mathsurround=1pt 
\title{Fold maps and immersions from the viewpoint of cobordism}

\thanks{The author was supported by Canon Foundation in Europe and has been supported
by JSPS}



\author{Boldizs\'ar Kalm\'{a}r}

\address{Faculty of Mathematics, Kyushu University, 6-10-1 Hakozaki, Higashi-ku, Fukuoka 812-8581, Japan}
\email{kalmbold@cs.elte.hu}

\begin{abstract}
We obtain complete geometric
invariants of cobordism classes of
oriented simple
fold maps of $(n+1)$-dimensional manifolds into an $n$-dimensional manifold $N^n$ 
in terms of immersions
with prescribed normal bundles.
We compute that for $N^n=\R^n$ the cobordism group of simple fold maps
is isomorphic to the 
direct sum of the $(n-1)$th
stable homotopy group of spheres and the $(n-1)$th stable homotopy group of the space $\RP^{\infty}$.
By using geometric invariants defined in the author's earlier works, 
we also describe the natural map of the simple fold cobordism group
to the fold cobordism group by natural homomorphisms between cobordism
groups of immersions. We also compute the ranks of the oriented (right-left) bordism groups of simple fold maps.
\end{abstract}

\maketitle


\section*{Introduction}\label{s:intro}
Fold maps of $(n+1)$-dimensional manifolds into $n$-dimensional manifolds 
are smooth singular maps which have the formula 
\[
f(x_1, \ldots, x_{n+1})=(x_1, \ldots, x_{n-1}, x_n^2 \pm x_{n+1}^2 )
\]
as a local normal form around each singular point.
Fold maps can be considered as the natural generalizations of Morse functions. 
 Let $f \co Q^{n+1} \to N^n$ be a fold map.
The set of singular points of the fold map $f$ is a two codimensional
smooth submanifold in the source manifold $Q^{n+1}$
and the fold map $f$ restricted to its singular points is a one codimensional immersion
into the target manifold $N^n$.
This immersion together with more detailed informations about the neighbourhood of the set of
singular points in the source manifold $Q^{n+1}$ can be used as a geometric invariant
 (see \cite{Kal6} and Section~\ref{foldgerms} of the present paper)
of fold cobordism classes (see Definition~\ref{cobdef}) of fold maps, and
by this way we obtain a geometric relation between fold maps and immersions via cobordisms.
In \cite{Kal4, Kal6}, we defined these invariants for negative codimensional\footnote{
If we have a map $f \co M^m \to P^p$ of an $m$-dimensional manifold
into a $p$-dimensional manifold, then the {\it codimension} of the map $f$
is the integer $p-m$.}  fold maps in full generality, and we used them to detect direct summands of the cobordism groups of negative codimensional fold maps.
In this paper, we study these invariants in the case of $-1$ codimensional fold maps with some additional restrictions about their  
singular fibers (for singular fibers, see \cite{Lev, Sa}).

{\it Simple fold maps} are fold maps with at most one 
singular point in each connected component of a singular fiber.
From this definition it follows that the only possible singular fibers 
whose singular points have sign ``$-$'' in the above normal form
are the disjoint unions of a finite number of ``figure eight''
singular fibers and circle components, provided that the source manifold is orientable or
the simple fold map is {\it oriented} 
(there is a consistent orientation of all fibers at their regular points). 
Simple fold maps have been studied, for example, by
Levine \cite{Lev}, 
Saeki \cite{Sa1, Sa2, Sa4},
Sakuma \cite{Saku} and 
Yonebayashi \cite{Yo}. 
The existence of a simple fold map on a manifold gives strong conditions about the structure of the manifold (for 
example, see the existence of simple fold maps on orientable $3$-manifolds \cite{Sa4}).
If we have a simple fold map 
of an oriented manifold 
or an oriented simple fold map, then 
the immersion of the singular set has trivial normal bundle in the target manifold
$N^n$, moreover there is a canonical trivialization corresponding to the number of
regular fiber components in a neighbourhood of a singular fiber.

The main result of this paper is that our geometric invariants describe completely
the set of cobordism classes of oriented simple fold maps of $(n+1)$-dimensional manifolds 
into an $n$-dimensional manifold $N^n$.
By using Pontryagin-Thom type construction, we prove that
the cobordism classes of oriented simple fold maps of $(n+1)$-dimensional 
manifolds into an $n$-dimensional manifold $N^n$ are in a natural bijection
with the set of stable homotopy classes of continuous maps of the one
point compactification of the manifold $N^n$ into the Thom-space of the trivial line bundle over the
space $\RP^{\infty}$. As a special case, we obtain that the
oriented cobordism group of simple fold
maps of oriented ${(n+1)}$-manifolds into $\R^n$ is isomorphic 
to the $n$th
stable homotopy group of the space $S^1 \vee S\RP^{\infty}$.
We also describe the natural homomorphism which maps a simple fold cobordism class 
to its fold cobordism class in terms of natural homomorphisms between cobordism groups
of immersions with prescribed normal bundles. In this way, we obtain results about 
the ``inclusion'' of the simple fold maps into the cobordism group of fold maps.
We have the analogous results about {\it bordisms} (see Definition~\ref{borddef}) of fold maps as well.

The paper is organized as follows.
In Section~\ref{s:jelol} we give several basic definitions and notations.
In Section~\ref{mainthms} we state our main results.
In Section~\ref{mainthmproof} we prove our main theorems.
In Section~\ref{specesetek} we give explicit descriptions of the ``inclusion'' of the simple
fold cobordism groups into the fold cobordism groups in low dimensions.
In Section~\ref{bord} we give analogous results about bordism groups of simple fold maps.
In Appendix we prove theorems about the singular fibers of simple fold maps.

The author would like to thank Prof. Andr\'as Sz\H{u}cs 
for the uncountable discussions 
and suggestions,
and Prof. Osamu Saeki for giving
useful observations and correcting the appearance of this paper.

\subsection*{Notations}
In this paper the symbol ``$\amalg$'' denotes the disjoint union, 
$\ga^k$ denotes the universal $k$-dimensional real vector bundle over $BO(k)$,
$\ep^k_X$ (or $\ep^k$) denotes the trivial $k$-plane bundle over the space $X$ (resp.\ over the point),
and the symbols $\csi^k$, $\eta^k$, etc. usually denote $k$-dimensional real vector bundles.
The symbols det$\csi^k$ and $T\csi^k$ denote the determinant line bundle 
and the Thom space of the bundle $\csi^k$, respectively.
The symbol $\imm^{\csi^k}_{N}(n-k,k)$ denotes
the cobordism group of $k$-codimensional immersions into 
an $n$-dimensional manifold $N$
whose normal bundles are induced from $\csi^k$ (this
group is isomorphic to the group $\{\dot N, T\csi^k \}$ \cite{We}, where
$\dot N$ denotes the one point compactification of the manifold $N$
and the symbol $\{X,Y \}$ denotes the group of stable homotopy classes of continuous
maps from the space $X$ to the space $Y$). 
In the case of $N^n = \R^n$ or $\csi^k = \ga^k$ we usually omit the symbols
$N$ or $\csi^k$ from the notation $\imm^{\csi^k}_{N}(n-k,k)$, respectively.
The symbol $\pi_n^s(X)$ (or $\pi_n^s$) denotes the $n$th stable homotopy group of the space $X$ (resp. spheres).
The symbol ``id$_A$'' denotes the identity map of the space $A$.
The symbol $\ep$ denotes a small positive number, and $\Omega_n$
denotes the oriented cobordism group of closed oriented $n$-dimensional manifolds.
All manifolds and maps are smooth of class $C^{\infty}$.

\section{Simple fold maps, fold cobordisms and geometric invariants}\label{s:jelol}

\subsection{Simple fold maps}
 
Let $Q^{n+1}$ and $N^n$ be smooth manifolds of dimensions $n+1$ and $n$ 
respectively. Let $p \in Q^{n+1}$ be a singular point of 
a smooth map $f \co Q^{n+1} \to N^{n}$. The smooth map $f$  has a {\it fold 
singularity} at the singular point $p$ if we can write $f$ in some local coordinates around $p$  
and $f(p)$ in the form 
\[  
f(x_1,\ldots,x_{n+1})=(x_1,\ldots,x_{n-1}, x_n^2 \pm x_{n+1}^2).
\] 
A smooth map $f \co Q^{n+1} \to N^{n}$ is called a {\it fold map} if $f$ has only 
fold singularities.

Singularities with the sign ``$+$'' or ``$-$'' in the above normal form are
called {\it definite} or {\it indefinite} fold singularities, respectively.

Let $S_1(f)$ and $S_0(f)$ denote the set of indefinite and definite fold singularities of $f$ in $Q^{n+1}$, respectively.
 Let $S_f$ denote the set $S_0(f) \cup S_1(f)$.
Note that the set $S_f$ is an ${(n-1)}$-dimensional submanifold of the manifold
$Q^{n+1}$.

If $f \co Q^{n+1} \to N^n$ is a fold map in general position, then 
the map $f$
restricted to the singular set $S_f$ is a general position
 codimension one immersion  into the target manifold $N^n$.

Since every fold map is in general position after a small perturbation, 
and we study maps under the equivalence relations {\it cobordism} and {\it bordism} 
(see Definitions~\ref{cobdef} and \ref{borddef} respectively),
in this paper we can restrict ourselves to studying fold maps which are 
in general position.
Without mentioning we suppose that a fold map $f$ is in general position.

A fold map $f \co Q^{n+1} \to N^{n}$ is called a {\it simple fold map} if every connected 
component of an arbitrary fiber of the map $f$ contains at most one singular point. 

A fold map $f \co Q^{n+1} \to N^{n}$ is called {\it framed} if the immersion $f|_{S_1(f)}$ of its indefinite
fold singular set is framed, i.e., its normal bundle is trivialized.

A fold map $f \co Q^{n+1} \to N^{n}$ is called {\it oriented}
if there is a chosen consistent orientation of all fibers at their regular points (for example, 
in the case of oriented source and target manifolds).

From the definition of simple fold maps it follows that an oriented simple fold map 
or a simple fold map of an oriented manifold 
can have only the disjoint union of a finite number of copies of the 
``figure eight'' and circles as an indefinite singular fiber (for singular fibers, see \cite{Lev,Sa}), see Figure~\ref{simplesing}.

\begin{figure}[ht]
\begin{center} 
\epsfig{file=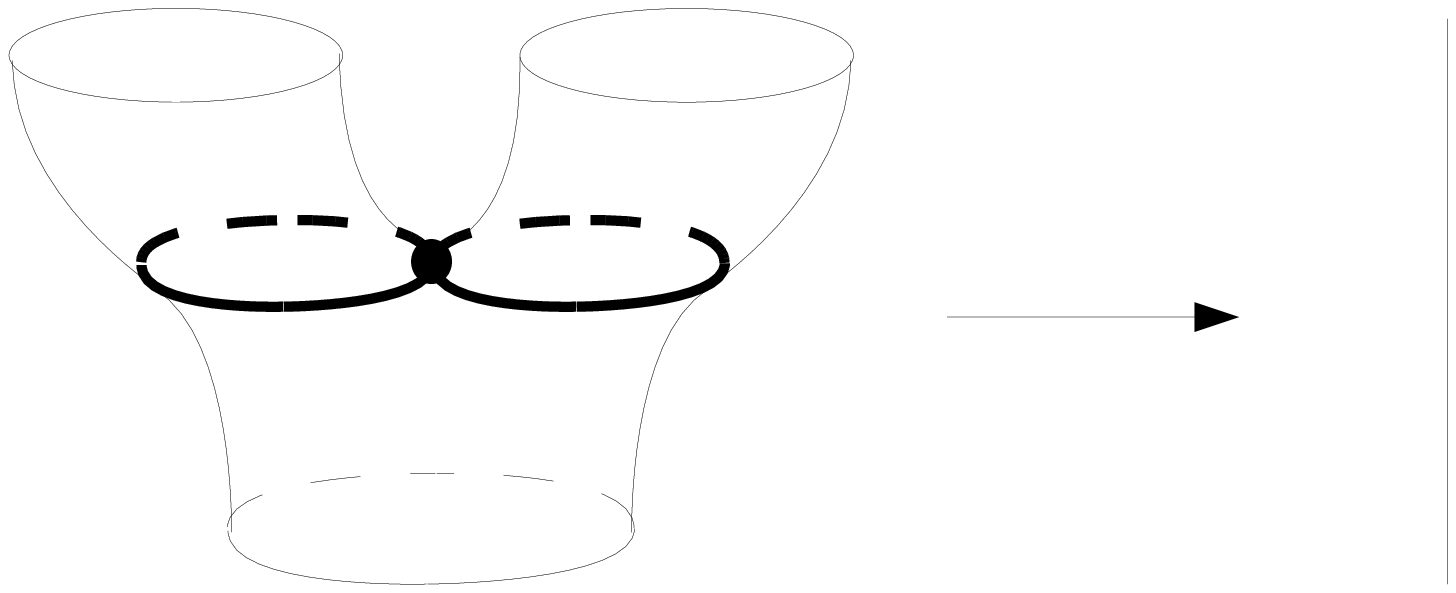, height=5cm} 
\caption{A representative of the singular fiber ``figure eight''.}
\end{center} 
\label{simplesing}  
\end{figure}

Note that an oriented simple fold map is framed in a canonical way since the immersion 
of its indefinite fold singular set has a
canonical trivialization corresponding to the number of
regular fiber components in a neighbourhood of a ``figure eight'' singular fiber.



\subsection{Cobordisms of simple fold maps}\label{kob}



\begin{defn}[Cobordism]\label{cobdef}
Two oriented fold maps (resp. two oriented simple fold maps) $f_i \co Q_i^{n+1} \to N^n$ $(i=0,1)$  
from closed $({n+1})$-dimensional manifolds $Q_i^{n+1}$ $(i=0,1)$ 
into an $n$-dimensional manifold $N^n$ are  
{\it cobordant} (resp. {\it simple cobordant}) if 
\begin{enumerate}
\item
there exists an oriented (resp.\ simple) fold map 
$F \co X^{n+2} \to N^n \times [0,1]$
from a compact $(n+2)$-dimensional 
manifold $X^{n+2}$,
\item
$\del X^{n+2} = Q_0^{n+1} \amalg Q_1^{n+1}$,
\item
${F |}_{Q_0^{n+1} \x [0,\ep)}=f_0 \x
{\mathrm {id}}_{[0,\ep)}$ and ${F |}_{Q_1^{n+1} \x (1-\ep,1]}=f_1 \x 
{\mathrm {id}}_{(1-\ep,1]}$, where $Q_0^{n+1} \x [0,\ep)$
 and $Q_1^{n+1} \x (1-\ep,1]$ are small collar neighbourhoods of $\del X^{n+2}$ with the
identifications $Q_0^{n+1} = Q_0^{n+1} \x \{0\}$ and $Q_1^{n+1} = Q_1^{n+1} \x \{1\}$,
\item
the orientations of the fold maps $f_0$, $f_1$ and $F$ under the above identifications are consistent.
\end{enumerate}

We call  the map $F$ a {cobordism} between $f_0$ and $f_1$.
\end{defn} 
This clearly defines an equivalence relation on the set of oriented fold maps (resp. simple fold maps)
from closed $({n+1})$-dimensional manifolds to an  
$n$-dimensional manifold $N^n$. Note that in the case of an oriented target manifold $N^n$
this cobordism relation coincides with the analogous cobordism relation of (simple) fold maps of oriented  
manifolds, where the manifold $X^{n+2}$ should be oriented and satisfy
$\del X^{n+2} = Q_0^{n+1} \amalg (-Q_1^{n+1})$.

Let us denote the cobordism classes of oriented fold maps into an $n$-dimensional manifold $N^n$
by $\CC ob_{\f}(N^n)$ and  
the cobordism classes of oriented simple fold maps into an $n$-dimensional manifold $N^n$
by $\CC ob_{\s}(N^n)$. 
We define a commutative semigroup operation in the usual way
on the set of cobordism classes $\CC ob_{\f}(N^n)$ and $\CC ob_{\s}(N^n)$
by the disjoint union,
which is a group operation in the case of $N^n=\R^n$.


\begin{defn}
Two oriented framed fold maps $f_i \co Q_i^{n+1} \to N^n$ $(i=0,1)$ are
{\it framed cobordant} if they are oriented cobordant by an oriented framed fold map $F \co X^{n+2} \to N^n \times [0,1]$ in the sense of Definition~\ref{cobdef}, such that the framing of $F$ is compatible with the framings of $f_0$ and
$f_1$. 
\end{defn}

We denote the framed cobordism classes 
 into an $n$-dimensional manifold $N^n$
by $\CC ob_{\f}^{\fr}(N^n)$, which clearly form a commutative semigroup under the disjoint union as operation.

\subsection{Bundles of fold germs}\label{foldgerms}
We summarize some properties
of usual germs $(\R^2,0) \to (\R,0)$.

Let us define the indefinite fold germ $g_{\mathrm {indef}} \co (\R^2,0) \to (\R,0)$ by 
$g_{\mathrm {indef}}(x,y) = x^2 - y^2$ and the definite fold germ
$g_{\mathrm {def}} \co (\R^2,0) \to (\R,0)$ by 
$g_{\mathrm {def}}(x,y) = x^2 + y^2$.

We say that a pair of diffeomorphism germs
$(\al \co (\R^2,0) \to (\R^2,0), \be \co (\R,0) \to (\R,0))$
is an {\it automorphism} of a germ $g \co (\R^2,0) \to (\R,0)$
if the equation $g \circ \al = \be \circ g$ holds.
We will work with bundles whose fibers and structure groups are germs and  
groups of automorphisms of germs, respectively.

If we have a fold map $f \co Q^{n+1} \to N^n$ with non-empty indefinite fold singular set, then
we have the commutative diagram
\begin{center}
\begin{graph}(6,2)
\graphlinecolour{1}\grapharrowtype{2}
\textnode {A}(0.5,1.5){$E(\csi_{\mathrm {indef}}^2(f))$}
\textnode {B}(5.5, 1.5){$E(\eta_{\mathrm {indef}}^1(f))$}
\textnode {C}(3, 0){$S_1(f)$}
\diredge {A}{B}[\graphlinecolour{0}]
\diredge {B}{C}[\graphlinecolour{0}]
\diredge {A}{C}[\graphlinecolour{0}]
\freetext (3,1.8){$E(\csi_{\mathrm {indef}}(f))$}
\freetext (0.7, 0.6){$\csi_{\mathrm {indef}}^2(f)$}
\freetext (5.3,0.6){$\eta_{\mathrm {indef}}^1(f)$}
\end{graph}
\end{center}
i.e., over the indefinite fold singular set $S_1(f)$, we have
\begin{enumerate}
\item
an $(\R^2,0)$ bundle denoted by $\csi_{\mathrm {indef}}^2(f) 
\co E(\csi_{\mathrm {indef}}^2(f)) \to S_1(f)$,
\item
an
$(\R,0)$ bundle denoted by $\eta_{\mathrm {indef}}^1(f)
\co E(\eta_{\mathrm {indef}}^1(f)) \to S_1(f)$ and
\item
a fiberwise map 
$E(\csi_{\mathrm {indef}}(f)) \co 
E(\csi_{\mathrm {indef}}^2(f)) 
 \to  E(\eta_{\mathrm {indef}}^1(f))$, which is equivalent on each fiber to 
 the indefinite germ $g_{\mathrm {indef}}$,
\end{enumerate}
and we say that 
we have an indefinite germ bundle\footnote{Note that total space $E(\csi_{\mathrm {indef}}(f))$ of this bundle is in fact a map, and is not a set.} 
over the indefinite singular set $S_1(f)$ 
denoted by $\csi_{\mathrm {indef}}(f) \co E(\csi_{\mathrm {indef}}(f)) \to S_1(f)$.

The {\it fiber} of the indefinite germ bundle $\csi_{\mathrm {indef}}(f)$ is
the indefinite germ $g_{\mathrm {indef}}$,
 the {\it base space} of the bundle $\csi_{\mathrm {indef}}(f)$ is the indefinite singular set $S_1(f)$
 and the {\it total space} of the bundle $\csi_{\mathrm {indef}}(f)$ is the fiberwise map 
 $E(\csi_{\mathrm {indef}}(f)) \co 
E(\csi_{\mathrm {indef}}^2(f)) 
 \to  E(\eta_{\mathrm {indef}}^1(f))$ between the total spaces of the bundles 
$\csi_{\mathrm {indef}}^2(f)$ and $\eta_{\mathrm {indef}}^1(f)$. We call 
the total space $E(\eta_{\mathrm {indef}}^1(f))$
the {\it target of the total space $E(\csi_{\mathrm {indef}}(f))$} of the indefinite germ bundle 
$\csi_{\mathrm {indef}}(f)$
and we call the bundle $\eta_{\mathrm {indef}}^1(f)$ the {\it target of the indefinite germ bundle 
$\csi_{\mathrm {indef}}(f)$}.

 By \cite{Jan, Szucs3, Wa} this bundle $\csi_{\mathrm {indef}}(f)$ is a locally trivial bundle 
 in a sense with a fiber $g_{\mathrm {indef}}$
 and an appropriate group of automorphisms
$(\al \co (\R^2,0) \to (\R^2,0), \be \co (\R,0) \to (\R,0))$ as structure group.
 The same holds for the definite singular set $S_0(f)$ and the definite germ
 $g_{\mathrm {def}}(x,y) = x^2 + y^2$.
 
Moreover, if we have an oriented fold map $f \co Q^{n+1} \to N^n$, 
then 
the elements of the structure groups of the above indefinite and definite 
germ bundles are automorphisms
$(\al \co (\R^2,0) \to (\R^2,0), \be \co (\R,0) \to (\R,0))$ 
where the diffeomorphisms $\al$ and $\be$
both preserve or both reverse the orientation, or in other words
they keep fixed a chosen consistent orientation at the regular points of the level curves.


If we have an indefinite germ bundle, then by \cite{Jan, Wa}
its structure group can be reduced to a maximal compact subgroup, namely to the dihedral group 
of order $8$ generated by the automorphisms $((x,y) \mapsto (x,-y), $id$_\R)$ and $((x,y) \mapsto (-y,x), -$id$_\R)$.
Furthermore, the automorphisms which keep fixed a chosen consistent orientation 
of the level curves of the indefinite germ
form a subgroup $\Z_2 \oplus \Z_2$, which is 
generated by $((x,y) \mapsto (y,x), -$id$_\R)$ and $((x,y) \mapsto (-y,-x), -$id$_\R)$.
In this group, we have the subgroup $\Z_2$ generated by
the automorphism $((x,y) \mapsto (-x,-y),\ $id$_\R)$.

If we have an oriented simple fold map 
 (or an oriented framed fold map),
then
the structure group of its indefinite germ bundle 
is reduced to this subgroup $\Z_2$, since other elements of $\Z_2 \oplus \Z_2$
cannot be extended to the singular fiber ``figure eight'' (resp. do not keep the trivialization of the 
target of the total space of the indefinite germ bundle).
Note that in the case of an oriented fold map without any framing in general this structure group can be more complicated than
$\Z_2$, see, for example, the singular fiber denoted by $\ii^3$ in \cite{Sa}\footnote{We note that it would not be 
very difficult to extend some of our results to oriented fold maps in an analogous way.
The main difference would be that since oriented fold maps can have more complicated singular fibers (e.g. the
singular fiber denoted by $\ii^3$ in \cite{Sa}), the symmetry group of the indefinite fold germ
$(x,y) \mapsto x^2 - y^2$ in the case of oriented fold maps is more complicated 
than in the case of oriented simple fold maps or oriented framed fold maps 
(where it is generated by the automorphism
$(x,y) \mapsto (-x,-y)$).}.

The automorphism group of the definite germ bundle of an oriented fold map
 can be reduced to a group whose elements
are of the form $((x,y) \to T(x,y),\ $id$_\R)$, where $T$ is an element 
of the group $SO(2)$. 

Now it follows that the targets of the universal oriented indefinite and definite germ bundles
are the line bundle 
$\eta_{\mathrm {indef}}^1 \co $det$(\ga^1 \x \ga^1) \to \RP^{\infty} \x \RP^{\infty}$ 
and the trivial line bundle $\eta_{\mathrm {def}}^1 \co \ep^1 \to \CP^{\infty}$, respectively.
Similarly, the target of the universal oriented simple (or framed) indefinite germ bundle is the trivial line bundle
over $\RP^{\infty}$.

The relation between the universal oriented indefinite and oriented simple (or framed) indefinite germ bundles  can be seen in the following bundle inclusion, which is induced by the inclusion of the above subgroup $\Z_2$ into the automorphism
group $\Z_2 \oplus \Z_2$, i.e., the diagonal map $\RP^{\infty} \to \RP^{\infty} \x \RP^{\infty}$.
\begin{equation}\label{bundlediag}
\begin{CD}
\ep_{\RP^{\infty}}^1 @>>> {\mathrm {det}}(\ga^1 \x \ga^1) \\
@VVV @VVV \\
\RP^{\infty} @>>> \RP^{\infty} \x \RP^{\infty}
\end{CD}
\end{equation}

%

\subsection{Cobordism invariants of oriented fold maps}

We define the homomorphisms
\[
{\mathcal I}_{\s}(N^n) \co \CC ob_{\s}(N^n) \to \imm^{\ep^1_{\RP^{\infty}}}_N(n-1,1)
\]
and
\[
{\mathcal D}_{\s}(N^n) \co \CC ob_{\s}(N^n) \to \imm^{\ep^1}_N(n-1,1)
\]
by
 mapping a cobordism class of an oriented simple fold map into the cobordism class of the immersion of its 
indefinite fold singular set 
with normal bundle induced from
the bundle $\ep^1 \to  \RP^{\infty}$,
and the cobordism class of the immersion of its 
definite fold singular set, respectively, and the homomorphisms
\[
{\mathcal I}_{\f}(N^n)  \co \CC ob_{\f}(N^n) \to \imm^{{\mathrm {det}}(\ga^1 \x \ga^1)}_N(n-1,1)
\]
and
\[
{\mathcal D}_{\f}(N^n) \co \CC ob_{\f}(N^n) \to \imm^{\ep^1_{\CP^{\infty}}}_N(n-1,1)
\]
by mapping a cobordism class of an oriented fold map into
 the cobordism class of the immersion of its 
indefinite fold singular set  with normal bundle induced from
the bundle $\eta_{\mathrm {indef}}^1 \co $det$(\ga^1 \x \ga^1) \to \RP^{\infty} \x \RP^{\infty}$,
and
 the cobordism class of the immersion of its 
definite fold singular set  with normal bundle induced from
 the bundle $\eta_{\mathrm {def}}^1 \co \ep^1 \to \CP^{\infty}$, respectively, 
 see the previous section.

Analogously, we define the homomorphisms
$
{\mathcal I}_{\f}^{\fr}(N^n)  \co \CC ob_{\f}^{\fr}(N^n) \to \imm^{\ep^1_{\RP^{\infty}}}_N(n-1,1)
$
and
$
{\mathcal D}_{\f}^{\fr}(N^n) \co \CC ob_{\f}^{\fr}(N^n) \to \imm^{\ep^1_{\CP^{\infty}}}_N(n-1,1)
$.


Note that the diagram~(\ref{bundlediag}) induces a 
homomorphism 
$$\imm^{\ep^1_{\RP^{\infty}}}_N(n-1,1) \to \imm^{{\mathrm {det}}(\ga^1 \x \ga^1)}_N(n-1,1)$$ and hence a
commutative diagram
\begin{equation}\label{invariinlc}
\begin{CD}
\CC ob_{\s}(N^n) @> {{\mathcal I}_{\s}(N^n)}  >>  \imm^{\ep^1_{\RP^{\infty}}}_N(n-1,1) \\
@VVV @VVV \\
\CC ob_{\f}(N^n)  @>{{\mathcal I}_{\f}(N^n)} >> \imm^{{\mathrm {det}}(\ga^1 \x \ga^1)}_N(n-1,1),
\end{CD}
\end{equation}
where the left vertical arrow is the natural homomorphism which maps a simple fold cobordism class to its fold cobordism class.

\section{Main results}\label{mainthms}

Now, we are ready to state our main theorems.

\begin{thm}\label{fotetel}
The semigroup homomorphism ${{\mathcal I}_{\s}(N^n)} $ is a semigroup isomorphism between the 
cobordism semigroup $\CC ob_{\s}(N^n)$ of oriented simple fold maps and the group 
$\imm^{\ep^1_{\RP^{\infty}}}_N(n-1,1)$. (The homomorphism ${{\mathcal I}_{\s}(\R^n)} $
is a group isomorphism.)
\end{thm}

\begin{cor}
The semigroup $\CC ob_{\s}(N^n)$ is a group.
\end{cor}

\begin{cor}\label{stablekov}
Let $n \geq 1$ and let $p$ be a prime number. Then,
\begin{enumerate}
\item
the cobordism group $\CC ob_{\s}(N^n)$ of oriented simple fold maps is isomorphic to
the group $\{ \dot N, S^1\} \oplus  \{ \dot N, S\RP^{\infty} \}$,
\item
the cobordism group $\CC ob_{\s}(\R^n)$ has no $p$-torsion if $p > {(n+2)}/{2}$,
and its $p$-torsion is $\Z_p$ if $p$ is odd and $n = 2p -2$ and
\item
the order of a simple fold map is always finite in the fold cobordism group
  $\CC ob_{\f}(\R^n)$ $(n>1)$.
\end{enumerate}
\end{cor}
\begin{proof}
(1) follows from the equalities 
$\CC ob_{\s}(N^n) = \imm^{\ep^1_{\RP^{\infty}}}_N(n-1,1)
= \{ \dot N, T\ep^1_{\RP^{\infty}} \} = \{ \dot N, S^1 \vee S\RP^{\infty}\} = 
\{ \dot N, S^1\} \oplus  \{ \dot N, S\RP^{\infty} \}$. (2) follows from well-known theorems about the prime-torsions of the groups 
 $\pi_{n-1}^s$ and  $\pi_{n-1}^s(\RP^{\infty})$. (3) follows from the fact that the groups 
 $\pi_{n-1}^s$ and  $\pi_{n-1}^s(\RP^{\infty})$ are finite for $n>1$.
\end{proof}

Let \[\phi_n^N \co \CC ob_{\s}(N^n) \to \CC ob_{\f}(N^n)\] 
($\phi_n$ in the case of $N^n = \R^n$)
denote the natural homomorphism
which maps a simple fold cobordism class into its fold cobordism class.


\begin{thm}\label{injdontes}
\noindent
\begin{enumerate}
\item
The simple fold cobordism 
group $\CC ob_{\s}(N^n)$ is a direct summand of the framed fold cobordism semigroup $\CC ob_{\f}^{\fr}(N^n)$.
\item
The direct summand $\{ \dot N, S^1\}$ of $\CC ob_{\s}(N^n)$
is mapped by $\phi_n^N$ isomorphically onto a direct summand of 
the fold cobordism semigroup $\CC ob_{\f}(N^n)$.
\item
If 
the natural forgetting homomorphism\footnote{Induced by the inclusion 
$\RP^{\infty} \hookrightarrow \RP^{\infty} \x \RP^{\infty}$,
$x \mapsto (*,x)$.
}
$$\imm^{\ep^1 \x \ga^1}_N(n-2,2) \to \imm^{\ga^1 \x \ga^1}_N(n-2,2)$$
 is injective, then so is $\phi_n^N$.
\item
If there exists a fold map from a not null-cobordant 
$(n+1)$-dimensional manifold into $N^n$, then
$\phi_n^N$ is not surjective.
\end{enumerate}
\end{thm}

We prove these in Section~\ref{geomkod}.

By Theorem~\ref{injdontes}, we obtain in Section~\ref{specesetek} the following.

\begin{prop}
For $n = 1, 2$ the homomorphism $\phi_n$ is an isomorphism,
$\phi_3$ is injective but not surjective, and 
for $n = 5, 6$ the homomorphism $\phi_n$ is injective.
\end{prop}

The following proposition will be useful in proving Theorem~\ref{fotetel}.

\begin{prop}\label{puncrepr}
Every element
of the simple fold cobordism group $\CC ob_{\s}(N^n)$ has a
representative $g \co Q^{n+1} \to N^n$  such that the source manifold 
$Q^{n+1}$ is the total space $S^2 \x S_{1}$ of a trivial two-sphere bundle  over $S_1$, where $S_{1}$ denotes 
the indefinite fold singular set of the fold map $g$, and 
the map $g$ restricted to any fiber $S^2$ is a composition $u \circ v$ of 
an embedding 
$u \co \R \to N^n$ and 
a Morse function $v \co S^2 \to \R$ with one indefinite and three definite critical points. 
\end{prop}

An easy corollary of this proposition is

\begin{cor}\label{nullkob}
If there exists an oriented simple fold map of an $(n+1)$-dimensional manifold 
$Q^{n+1}$ into an (oriented)  $n$-dimensional manifold, then 
the manifold $Q^{n+1}$ is (oriented) null-cobordant.
\end{cor}

Compare Corollary~\ref{nullkob} with \cite[Proposition~3.12]{Sa1}.

\begin{rem}\label{ori3mfdnullcob}
As an application of the surjectivity of 
the homomorphism $\phi_2 \co \CC ob_{\s}(\R^2) \to \CC ob_{\f}(\R^2)$,
we have a nice argument to show that every orientable
3-manifold is oriented null-cobordant.
Namely let $M^3$ be an orientable 3-manifold. Let $f \co M^3 \to
\R^2$ be a smooth stable map. By \cite{Lev1} cusp singularities can be eliminated by homotopy
(but now elimination by cobordism is enough). 
Hence, we obtain a fold map $f_1 \co M_1^3 \to
\R^2$ whose source manifold is oriented cobordant to $M^3$.
By the surjectivity of $\phi_2$, we have a simple fold map $f_2 \co M_2^3 \to \R^2$
whose source manifold is oriented cobordant to $M^3$.
Now by Proposition~\ref{puncrepr}
$M_2^3$ is oriented cobordant to a trivial $2$-sphere bundle, and therefore the
manifold $M^3$ is oriented null-cobordant.
(Compare with \cite{CoTh}.)
\end{rem}



As a corollary in Section~\ref{ketto}, we give another proof to the 
main result of \cite{Kal2} together with some geometric invariants,
 namely we have

\begin{thm}\label{3to2ujbiz}
The oriented cobordism group $\CC ob_{\f}(\R^2)$ of fold maps of 
$3$-manifolds into the plane is isomorphic to
$\Z_2 \oplus \Z_2$ through the homomorphism $\CC ob_{\f}(\R^2) \to \imm(1, 1) \oplus \imm(1, 1)$
which assigns to a class $[f]$ the sum 
$[f|_{S_0(f)}] \oplus [f|_{S_1(f)}]$
of the immersions of the definite and the indefinite fold singular sets.
\end{thm}

By the surjectivity of the natural homomorphism 
$\phi_2 \co \CC ob_{\s}(\R^2) \to \CC ob_{\f}(\R^2)$,
we have

\begin{cor}
Non-simple singular fibers of an oriented fold map $f \co M^3 \to \R^2$
can be eliminated by cobordism. 
\end{cor}

Analogous results about the oriented bordism groups of simple fold maps $Bor_{\s}(n+1,-1)$
can be found in Section~\ref{bord}. Moreover, we have the following.

\begin{thm}
The rank of the oriented simple bordism group $Bor_{\s}(n+1,-1)$ is equal to
$\rank{\Omega_n} + \sum_{q = 0}^{n} \rank{\Omega_q}$.
\end{thm}

%

\section{Proof of main theorems}\label{mainthmproof}

First let us prove Proposition~\ref{puncrepr}. 
\begin{proof}[Proof of Proposition~\ref{puncrepr}]

For a given simple fold map $f \co Q_0^{n+1} \to N^n$
let us construct a fold map $g \co Q^{n+1} \to N^n$ by \cite[Proposition~3.4]{Kal2}. It is easy to see that the fold map $g$ is also simple, and if the fold map $f$ is oriented, then so is $g$. 
Moreover, we can suppose that the indefinite fold singular set of $f$ is not empty (otherwise $f$ is null-cobordant as one can see easily), and by \cite[Proposition~3.4]{Kal2} it coincides with the
indefinite fold singular set of $g$.

Let $V$ denote a fiber of the normal bundle of the $g$-image $g(S_1)$ 
of the indefinite fold singular set $S_1$ of the map $g$ in the manifold $N^n$.
Let us suppose that the arc $V$ is transversal to the immersions into $N^n$ of the definite and indefinite fold singular sets of $g$, and intersects the set $g(S_1)$ only in one point $p$ which is not a multiple point of the immersion $g |_{S_1}$.

It is enough to remark that since the fold map $g$ is oriented and simple,
the arc $V$ can be chosen so that the component of the preimage $g^{-1}(V)$ 
 containing the ``figure eight" of the simple singular fiber $(g^{-1}(V), g^{-1}(\{p\})) \to (V,p)$
is a two dimensional sphere $S^2$, and the restriction 
of $g$ to this sphere is a Morse function with three definite and one indefinite critical points. 

Hence the source manifold $Q^{n+1}$ of the fold map $g$
is an $S^2$-bundle over $S_1$, the map $g$ restricted to any fiber $S^2$ is a Morse function
with three definite and one indefinite critical points, and the map $g$ is a composition of a (possibly non-trivial) family $Q^{n+1} \to \R \x S_1$ of such Morse functions parametrized by the indefinite fold singular set $S_1$ and an immersion $\R \x S_1 \to N^n$ onto the regular neighbourhood of $g(S_1)$.

Moreover, it is easy to see that the structure group of this sphere bundle can be reduced to $Z_2$ (essentially because of the ``figure eight'' subbundle of the sphere bundle coming from the Morse function family\footnote{
For a more precise argument, see  Theorem~\ref{strcsopveges} in Appendix.
}),
where the non-trivial element of $\Z_2$ acts on the fiber $S^2$ as a $180$ degree rotation,
hence this $S^2$-bundle is trivial. 
\end{proof}





Now let us prove Theorem~\ref{fotetel}.

\begin{proof}[Proof of Theorem~\ref{fotetel}]\label{geomkod}

Here, we give the proof only in the case of fold maps into $\R^n$ but the techniques and the results
are the same when the target manifold is an arbitrary $n$-dimensional manifold.

First, we rephrase our original problem about the cobordism group of simple fold maps.
By Proposition~\ref{puncrepr}, we can always choose a representative $g$ for an arbitrary simple fold cobordism class in $\CC ob_{\s}(\R^n)$, such that $g$ is equal to the composition 
\[
\begin{CD}
EG \x_G S^2  @>\varrho>> EG \x \R @> i >> \R^n
\end{CD}
\]
where 
\begin{enumerate}[1)]
\item
the map $\varrho \co EG \x_G S^2 \to EG \x \R$ is a family\footnote{Which depends on $g$. This family is locally trivial in a sense, hence it is a ``bundle'' over the indefinite singular set $S_1(g)$ with a Morse function as fiber and with structure group $\Z_2$, for details, see Appendix.} of the Morse function $h \co S^2 \to \R$ of the $2$-sphere $S^2$ parametrized by the indefinite singular set $S_1(g)$, where the Morse function $h$ has three definite and one indefinite critical points,
\item 
$G$ is the structure group of this family\footnote{Since this family is a ``locally trivial bundle'', it has ``structure group'', see Appendix.}, which is 
isomorphic to $\Z_2$,
where the non-trivial element of $\Z_2$ acts on the ``fiber'' 
$h \co S^2 \to \R$ as a rotation of 180 degree on $S^2$ and identically on $\R$ keeping fixed the indefinite critical point and one definite critical point, and
\item
the map $i \co EG \x \R \to \R^n$
is an immersion onto a tubular neighbourhood of $g(S_1(g))$ such that $i |_{EG \x \{0\}}$ corresponds to $g |_{S_1(g)}$.
\end{enumerate}

Moreover, by the same process, a simple fold cobordism $F \co X^{n+2} \to \R^n \x [0,1]$ between two such representatives $g_0$ and $g_1$ can be chosen to have a form 
\[
\begin{CD}
EG \x_G S^2  @> \Upsilon>> EG \x \R @> j >> \R^n \x [0,1]
\end{CD}
\]
with the analogous properties, i.e. it is a composition of a family $\Upsilon$ of the Morse function $h$ with structure group $\Z_2$ parametrized by the indefinite singular set of $F$ and an immersion $j$, compatible with $g_0$ and $g_1$ near the boundary $\del X^{n+2}$.

Hence, we have a homomorphism
\begin{equation}\label{muhomom}
\mu_n \co \CC ob_{\s}(\R^n) \to \imm^{\ep^1_{\RP^{\infty}}}(n-1,1)
\end{equation} 
by sending a cobordism class $[g]$ $(=[i \circ \varrho])$ to the immersion of the indefinite singular set $S_1(g)$ with normal bundle induced from the trivial line bundle over $\RP^{\infty}$ $(=B\Z_2)$ according to the principal $\Z_2$ bundle over $S_1(g)$ obtained from the $\Z_2$-symmetry of the family $\varrho$.

\begin{prop}
The homomorphism $\mu_n$ is an isomorphism.
\end{prop}
\begin{proof}
The homomorphism $\mu_n$ is surjective because for a given immersion $i \co M^{n-1} \to \R^n$ with normal bundle $\nu$ induced from the bundle $\ep^1 \to  \RP^{\infty}$, we can construct a family 
$\varrho \co E\Z_2 \x_{\Z_2} S^2  \to E\Z_2 \x \R$ of the Morse function $h \co S^2 \to \R$ which corresponds to the principal $\Z_2$ bundle over $M^{n-1}$ obtained from the inducing map 
$M^{n-1} \to  \RP^{\infty}$ of the normal bundle $\nu$.
The homomorphism $\mu_n$ is clearly injective as well.
\end{proof}

The observation that the homomorphism $\mu_n$ coincides with the homomorphism ${{\mathcal I}_{\s}(\R^n)}$
 completes the proof of Theorem~\ref{fotetel}.
\end{proof}


\subsection{Representing the natural homomorphism $\CC ob_{\s}(N^n) \to \CC ob_{\f}(N^n)$}

Let \[\phi_n^N \co \CC ob_{\s}(N^n) \to \CC ob_{\f}(N^n)\] 
($\phi_n$ in the case of $N^n = \R^n$)
denote the natural homomorphism
which maps a simple fold cobordism class into its fold cobordism class.

We define the homomorphism 
\[
\theta_{\f}^N \co \imm^{{\mathrm {det}}(\ga^1 \x \ga^1)}_N(n-1,1)
\to \imm_N(n-1,1) \oplus \imm^{\ga^1 \x \ga^1}_N(n-2,2)
\]
as follows.
Let $[h \co M^{n-1} \to N^n]$ be an element of $\imm^{{\mathrm {det}}(\ga^1 \x \ga^1)}_N(n-1,1)$.
Then the natural forgetting homomorphism 
$\iota \co \imm^{{\mathrm {det}}(\ga^1 \x \ga^1)}_N(n-1,1) \to \imm_N(n-1,1)$ gives 
a class $[h_1] = \iota([h])$ in $\imm_N(n-1,1)$.
Furthermore, we induce the normal bundle $\nu(h)$ of the immersion $h$ from
the bundle
${\mathrm {det}}(\ga^1 \x \ga^1) \to \RP^{\infty} \x \RP^{\infty}$,
i.e., we have the fiberwise isomorphism
\[
\begin{CD}
\nu(h) @>>> {\mathrm {det}}(\ga^1 \x \ga^1) \\
@VVV @VVV \\
M^{n-1} @> \tilde \nu_h >> \RP^{\infty} \x \RP^{\infty}.
\end{CD}
\]
By composing the map $\tilde \nu_h$ with the projection to the first factor 
$\pi_1 \co \RP^{\infty} \x \RP^{\infty} \to \RP^{\infty}$, 
we obtain a line bundle $\la^1$ over $M^{n-1}$ by the diagram
\[
\begin{CD}
\la^1 @>>> \ga^1 \\
@VVV @VVV \\
M^{n-1} @> \pi_1 \circ \tilde \nu_h >> \RP^{\infty}.
\end{CD}
\]
This line bundle $\la^1$ over the $(n-1)$-dimensional manifold $M^{n-1}$
gives an $({n-2})$-dimensional manifold $L^{n-2}$ and an embedding 
$P_{\la^1} \co L^{n-2} \hookrightarrow M^{n-1}$,
which represents the Poincar\'e dual to the first Stiefel-Whitney class $w_1(\la^1)$. 
Hence, we obtain an immersion $h_2 = h \circ P_{\la^1} \co L^{n-2} \to N^n$ with
normal bundle induced from the bundle $\ga^1 \x \ga^1$ such that the normal bundle of $h$ corresponds to the first $\ga^1$ term and the normal bundle of $P_{\la^1}$ corresponds to the second $\ga^1$ term, i.e., an 
element of
$\imm^{\ga^1 \x \ga^1}_N(n-2,2)$. Now let us define the homomorphism
$$\theta_{\f}^N \co \imm^{{\mathrm {det}}(\ga^1 \x \ga^1)}_N(n-1,1)
\to \imm_N(n-1,1) \oplus \imm^{\ga^1 \x \ga^1}_N(n-2,2)$$ by
$\theta_{\f}^N([h]) = [h_1] \oplus [h_2]$.

We define the homomorphism 
$$\theta_{\s}^N \co  \imm^{\ep^1_{\RP^{\infty}}}_N(n-1,1)  \to 
\imm^{\ep^1}_N(n-1,1) \oplus \imm^{\ep^1 \x \ga^1}_N(n-2,2)$$
similarly, i.e., for an element $[g \co M^{n-1} \to N^n]$ of the group 
$\imm^{\ep^1_{\RP^{\infty}}}_N(n-1,1)$ let $[g_1]$ be the image of $[g]$ under the
natural forgetting homomorphism $\imm^{\ep^1_{\RP^{\infty}}}_N(n-1,1) \to  \imm^{\ep^1}_N(n-1,1)$, 
and let $[g_2] =  [g \circ P_{\ka^1} \co K^{n-2} \to N^n]$, where the embedding 
$$P_{\ka^1} \co K^{n-2}  \hookrightarrow M^{n-1}$$ represents the 
Poincar\'e dual to the first Stiefel-Whitney class of the
line bundle $\ka^1 \to M^{n-1}$ obtained from the universal line bundle $\ga^1 \to \RP^{\infty}$ by the 
map $M^{n-1} \to  \RP^{\infty}$ which induces the normal
bundle of the immersion $g \co M^{n-1} \to N^n$ from $\ep^1_{\RP^{\infty}}$.

Let $\ga_n^N$ denote the natural forgetting homomorphism
\[
\imm^{\ep^1}_N(n-1,1) \oplus \imm^{\ep^1 \x \ga^1}_N(n-2,2) 
\to \imm_N(n-1,1) \oplus \imm^{\ga^1 \x \ga^1}_N(n-2,2)
\]
induced by the inclusion 
$\RP^{\infty} \hookrightarrow \RP^{\infty} \x \RP^{\infty}$,
$x \mapsto (*,x)$. 

By the above constructions and by the diagram~\eqref{invariinlc},
we have the commutative diagram
\begin{equation}\label{invaridiag}
\begin{CD}
\CC ob_{\s}(N^n) @> {{\mathcal I}_{\s}(N^n)}>>  \imm^{\ep^1_{\RP^{\infty}}}_N(n-1,1) @> \theta_{\s}^N >> \imm^{\ep^1}_N(n-1,1) \oplus \imm^{\ep^1 \x \ga^1}_N(n-2,2) \\
@VV \phi_n^N V @VVV @V \ga_n^N  VV \\
\CC ob_{\f}(N^n) @>{{\mathcal I}_{\f}(N^n)} >> \imm^{{\mathrm {det}}(\ga^1 \x \ga^1)}_N(n-1,1) @> \theta_{\f}^N >> \imm^{}_N(n-1,1) \oplus \imm^{\ga^1 \x \ga^1}_N(n-2,2).
\end{CD}
\end{equation}

\begin{lem}
The homomorphism $\theta_{\s}^N$ is an isomorphism.
\end{lem}
\begin{proof}
By the construction of $\theta_{\s}^N$ the statement follows easily. In fact, the homomorphism $\theta_{\s}^N$
yields the isomorphism $\{ \dot N, T\ep^1_{\RP^{\infty}} \} = \{ \dot N, S^1 \vee S\RP^{\infty} \} = 
\{ \dot N, S^1 \} \oplus \{ \dot N, S\RP^{\infty} \}$ under
the identifications given by the Pontryagin-Thom construction.
\end{proof}


\subsection{Proof of Theorem~\ref{injdontes}}\label{invari}

\begin{proof}[Proof of Theorem~\ref{injdontes}]

We define a homomorphism $$\psi(N^n) \co \imm^{\ep^1_{\RP^{\infty}}}_N(n-1,1) \to 
\CC ob_{\f}^{\fr}(N^n)$$ so that the composition ${{\mathcal I}_{\f}^{\fr}(N^n)} \circ \psi(N^n)  \co
 \imm^{\ep^1_{\RP^{\infty}}}_N(n-1,1) \to  \imm^{\ep^1_{\RP^{\infty}}}_N(n-1,1)$ is the identity map, and
the composition $$\imm^{\ep^1_{\RP^{\infty}}}_N(n-1,1) \plto{\psi(N^n)} \CC ob_{\f}^{\fr}(N^n) \plto{\al} \CC ob_{\f}(N^n)
\plto{{\mathcal D}_{\f}(N^n)}  \imm^{\ep^1_{\CP^{\infty}}}_N(n-1,1),$$where $\al$ is the natural forgetting
homomorphism,
maps the direct summand $\imm^{\ep^1}_N(n-1,1)$ of $\imm^{\ep^1_{\RP^{\infty}}}_N(n-1,1)$ 
isomorphically onto the direct summand 
$\imm^{\ep^1}_N(n-1,1)$ of $\imm^{\ep^1_{\CP^{\infty}}}_N(n-1,1)$. 
Namely, let $\psi(N^n)$ be the inverse of \eqref{muhomom}\footnote{The homomorphism 
\eqref{muhomom} was defined only in the case of $N^n = \R^n$ but obviously we have its version for an arbitrary
manifold $N^n$.}, which maps an element in $\imm^{\ep^1_{\RP^{\infty}}}_N(n-1,1)$  into the cobordism class of the
corresponding representative $g$ (see the proof of Theorem~\ref{fotetel}). This gives us 
(1) of Theorem~\ref{injdontes}, and also (2) if we observe in the proof of Theorem~\ref{fotetel} that the three immersions of the definite fold singular set
of the representative $g$ (which correspond to the three definite critical points of the Morse function $h$),
where the structure group of the Morse function family is trivial, are parallel to the immersion $g|_{S_1(g)}$ of the indefinite 
fold singular set of $g$ and their sum represents the inverse of $[g|_{S_1(g)}]$ (see also the proof of \cite{Kal4}).
Because of the diagram~\eqref{invaridiag} and since the homomorphisms 
${{\mathcal I}_{\s}(N^n)}$ and $\theta_{\s}^N$ are isomorphisms, by using (2), we obtain (3). Finally, (4)
is obvious from Proposition~\ref{nullkob} (which we proved independently from Theorem~\ref{injdontes}).
\end{proof}

\section{Special cases}\label{specesetek}

Now, we study in low dimensions the homomorphism $\phi_n \co \CC ob_{\s}(\R^n) \to \CC ob_{\f}(\R^n)$
which maps a simple fold cobordism class into its fold cobordism class.

\subsection{$n = 1$}

By \cite{IS} the cobordism group
 $\CC ob_{\f}(\R)$ is isomorphic to $\Z$. 
 By Theorem~\ref{fotetel} the cobordism group $\CC ob_{\s}(\R)$ is 
 isomorphic to $\Z$ (see also \cite{Sa6}) and it is easy to see that 
 the homomorphism $\phi_1 \co \CC ob_{\s}(\R) \to \CC ob_{\f}(\R)$
is surjective hence it is an isomorphism.

\subsection{$n = 2$}\label{ketto}

By \cite{Kal2} the cobordism group
$\CC ob_{\f}(\R^2)$ is isomorphic to $\Z_2 \oplus \Z_2$.
Here, we give another proof to this isomorphism together with some geometric invariants.

Let $\iota_2 \co \CC ob_{\f}(\R^2) \to \imm(1,1)$ and
$\de_2 \co \CC ob_{\f}(\R^2) \to \imm(1,1)$ denote the 
homomorphisms which map a fold cobordism class
into the immersion of its indefinite and
definite fold singular set respectively.

Note that the cobordism group $\imm(1,1)$ is isomorphic to $\Z_2$, and
by Theorem~\ref{fotetel} the cobordism group
$\CC ob_{\s}(\R^2)$ is isomorphic to $\Z_2 \oplus \Z_2$.

\begin{lem}\label{elsolem}
The homomorphism $(\iota_2 \oplus \de_2) \circ \phi_2
\co \CC ob_{\s}(\R^2) \to \Z_2 \oplus \Z_2$ is an isomorphism.
\end{lem}
\begin{proof}
It is enough to show that $(\iota_2 \oplus \de_2) \circ \phi_2$
is surjective. This is true because it is easy to construct 
simple fold maps of oriented 3-manifolds into the plane 
whose indefinite and definite fold singular sets are immersed 
into $\R^2$ in a prescribed way (see Proposition~\ref{puncrepr}).
\end{proof}

\begin{lem}\label{masodiklem}
The homomorphism $\phi_2 \co \CC ob_{\s}(\R^2) \to \CC ob_{\f}(\R^2)$ is surjective.
\end{lem}
\begin{proof}
By the classification of singular fibers of stable fold maps into $3$-manifolds \cite{Lev, Sa}
it is easy to show the surjectivity as follows.

Let ${\si}_{\SS} \co {s}_{\SS} \to (-\ep, \ep)^3$ be a representative of a singular fiber $\SS$.
Then the fold map 
${\si}_{\SS}|_{{\si}_{\SS}^{-1}(\del [-\ep/2,\ep/2]^3)} \co {\si}_{\SS}^{-1}(\del [-\ep/2,\ep/2]^3) \to \del [-\ep/2,\ep/2]^3$ (after applying 
\cite[Proposition~3.4]{Kal2})
represents a null-cobordant element in the cobordism group $\CC ob_{\f}(\R^2)$.

In this way, we obtain that the singular fiber ${\iii^4}$ shows a cobordism between a fold map with 
non-simple singular fibers of type $\ii^2$ and a simple fold map, and
the singular fiber ${\iii^6}$ shows a cobordism between a fold map with
non-simple singular fibers of types $\ii^2$, $\ii^3$ and a fold map
with only
non-simple singular fibers of type $\ii^2$ (for the notations $\ii^2$, $\ii^3$, $\iii^4$ and $\iii^6$, see \cite{Sa}),
Hence a fold map which has non-simple singular fibers 
is cobordant to a simple fold map, details are left to the reader.
\end{proof}

By the above two lemmas, we can give a proof different from that given in \cite{Kal2} to the
computation of the cobordism group $\CC ob_{\f}(\R^2)$.

\begin{proof}[Proof of Theorem~\ref{3to2ujbiz}]
By Lemma~\ref{elsolem} $\phi_2$ is injective and so by
Lemma~\ref{masodiklem} it is an isomorphism.
Hence the homomorphism 
\[
\iota_2 \oplus \de_2 \co \CC ob_{\f}(\R^2) \to \imm(1,1) \oplus \imm(1,1)
\] 
which maps a fold cobordism class $[f]$ into the direct sum $[f |_{S_1(f)}] \oplus [f |_{S_0(f)}]$
is also an isomorphism.
\end{proof}

\subsection{$n = 3$}


By Theorem~\ref{injdontes},
we have to study the natural homomorphism
\[
\ga_{3,2} \co \imm^{\ep^1 \x \ga^1}(1,2) \to \imm^{\ga^1 \x \ga^1}(1,2).
\]

If $\ga_{3,2}$ is injective, then the homomorphism 
$\phi_3 \co \CC ob_{\s}(\R^3) \to \CC ob_{\f}(\R^3)$ is also injective.

The group $\imm^{\ep^1 \x \ga^1}(1,2)$ is isomorphic to $\Z_2$.
Let $r \co \imm^{\ga^1 \x \ga^1}(1,2) \to \imm(2,1)$ be a homomorphism such that
for an element $[s] \in \imm^{\ga^1 \x \ga^1}(1,2)$ an immersed surface representing
$r([s])$ is obtained by putting a ``figure eight'' in each fiber of the $2$-dimensional normal bundle of
the representative $s$ in $\R^3$, which is invariant under the structure group of the bundle $\ga^1 \x \ga^1$.

Since the composition $r \circ \ga_{3,2}$ is injective,
as one can check easily, by Theorem~\ref{injdontes}, we get that 
the homomorphism $\phi_3$ is injective. 


By Theorem~\ref{injdontes} and the existence of a fold map from a not null-cobordant
oriented 4-manifold \cite{Sa}, we have that the homomorphism $\phi_3$ is not surjective.

%
%

\subsection{$n = 5$}

In this case the homomorphism $\ga_{5,2}$ is also injective
because of the following. By \cite{Hu}
the group $\imm^{\ep^1 \x \ga^1}(3,2) \cong \imm^{\ga^1}(3,1) \cong \Z_2$ is mapped
by the forgetting homomorphism $\imm^{\ep^1 \x \ga^1}(3,2) \to \imm^{}(3,2)$
into the group $\imm^{}(3,2) \cong \Z_2$ injectively,
hence the first forgetting homomorphism in the composition 
$\imm^{\ep^1 \x \ga^1}(3,2) \to \imm^{\ga^1 \x \ga^1}(3,2) \to \imm^{}(3,2)$
is injective.
Therefore the homomorphisms $$\ga_{5,2} \co \imm^{\ep^1 \x \ga^1}(3,2) \to \imm^{\ga^1 \x \ga^1}(3,2)$$
and $\phi_5$ are injective.

\subsection{$n = 6$}

The cobordism group $\CC ob_{\s}(\R^6)$ vanishes,
hence $\phi_6$
is clearly injective.

%
%


%


\section{Bordisms of fold maps}\label{bord}

\begin{defn}[Bordism]\label{borddef}
Two fold maps (resp. simple fold maps) $f_i \co Q_i^{n+1} \to N_i^n$   
of closed oriented ${(n+1)}$-dimensional manifolds $Q_i^{n+1}$ 
into closed oriented $n$-dimensional manifolds $N_i^n$ $(i=0,1)$ are  
{bordant} (resp. {\it simple bordant}) if 
\begin{enumerate}
\item
there exists a fold map (resp. simple fold map)
$F \co X^{n+2} \to Y^{n+1}$ 
from a compact oriented $(n+2)$-dimensional 
manifold $X^{n+2}$ to a compact oriented $(n+1)$-dimensional 
manifold $Y^{n+1}$,
\item
$\del X^{n+2} = Q_0^{n+1} \amalg (-Q_1)^{n+1}$, $\del Y^{n+2} = N_0^{n+1} \amalg (-N_1)^{n+1}$ and
\item
${F |}_{Q_0^{n+1} \x [0,\ep)}=f_0 \x
id_{[0,\ep)}$ and ${F |}_{Q_1^{n+1} \x (1-\ep,1]}=f_1 \x id_{(1-\ep,1]}$, where $Q_0^{n+1} \x [0,\ep)$
 and $Q_1^{n+1} \x (1-\ep,1]$ are small collar neighbourhoods of $\del X^{n+2}$ with the
identifications $Q_0^{n+1} = Q_0^{n+1} \x \{0\}$, $Q_1^{n+1} = Q_1^{n+1} \x \{1\}$. 
\end{enumerate}

We call  the map $F$ a {bordism} between $f_0$ and $f_1$.
\end{defn}

We can define a commutative group operation on the set of bordism classes
by $[f \co  Q_1^{n+1} \to N_1^n] + [g \co Q_2^{n+1} \to N_2^n] = f \amalg g \co  Q_1^{n+1} \amalg Q_2^{n+1} \to N_1^n \amalg N_2^n$
in the usual way.

We obtain analogous theorems about the bordism group of simple fold maps denoted by
$Bor_{\s}(n)$ and the natural homomorphism $\phi_n^{Bor} \co Bor_{\s}(n) \to Bor_{\f}(n)$ into
the bordism group of fold maps denoted by $Bor_{\f}(n)$.
($\phi_n^{Bor}$ maps a simple fold bordism class into its fold bordism class.)

Let $B\imm^{\csi^k}(n-k,k)$ denote the usual bordism group
of $k$-codimensional immersions into closed oriented $n$-dimensional manifolds, whose normal bundles
can be induced from the vector bundle $\csi^k$. Note that this group
$B\imm^{\csi^k}(n-k,k)$ is isomorphic to the $n$th oriented bordism group 
$\Omega_{n}(\Gamma_{\csi^k})$
of the classifying space $\Gamma_{\csi^k}$ for such immersions \cite{Sch, Sz1}.

\begin{thm}\label{borizom}
The bordism group $Bor_{\s}(n)$ of simple fold maps is isomorphic to
the group $B\imm^{\ep^1}(n-1,1) \oplus B\imm^{\ep^1 \x \ga^1}(n-2,2)$.\qed
\end{thm}

In other words the bordism group $Bor_{\s}(n)$ is
isomorphic to the direct sum $\Omega_{n}(\Gamma_{\ep^1}) \oplus \Omega_{n}(\Gamma_{\ep^1 \x \ga^1})$.

Let 
$
\ga_{n,2}^{Bor} \co B\imm^{\ep^1 \x \ga^1}(n-2,2) \to B\imm^{\ga^1 \x \ga^1}(n-2,2)
$
denote the natural forgetting homomorphism, and 
let $\pi_{n,2}^{Bor} \co Bor_{\s}(n+1,-1) \to B\imm^{\ep^1 \x \ga^1}(n-2,2)$
be the isomorphism of Theorem~\ref{borizom} composed with the projection to the second factor.

\begin{thm}
If two simple fold bordism classes $[f]$ and $[g]$ 
are mapped into distinct elements by the natural homomorphism
$\ga_{n,2}^{Bor} \circ \pi_{n,2}^{Bor}$,
then $[f]$ and $[g]$ are not fold bordant.
If $\ga_{n,2}^{Bor}$ is injective, then so is $\phi_n^{Bor}$.

If there exists a fold map from a not null-cobordant 
$(n+1)$-dimensional manifold into a closed oriented $n$-dimensional manifold, then
$\phi_n^{Bor}$ is not surjective.\qed
\end{thm}

%
%
%
%

A theorem \cite{CoFl} that can be applied here is that
the rank of the bordism groups of any space $X$ can be computed by
$\Omega_*(X) \otimes \Q = H_*(X;\Q) \otimes \Omega_*$, i.e., 
$\Omega_n(X) \otimes \Q = \bigoplus_{p + q =n} H_p(X;\Q) \otimes \Omega_q$.

By \cite{Sz1} the group $\Omega_{n}(\Gamma_{\ep^1}) \otimes \Q$ is isomorphic to
$\bigoplus_{i=0}^n \Omega_i \otimes \Q$. By \cite{We} the cobordism group 
$\imm^{\ep^1 \x \ga^1}(n-2,2) \cong
\pi_n(\Gamma_{\ep^1 \x \ga^1})$ is $2$-primary and finite, hence by standard arguments 
the reduced 
cohomology ring
$\Bar H_*(\Gamma_{\ep^1 \x \ga^1};\Z)$ is also $2$-primary and finite.
Therefore the group $\Omega_{n}(\Gamma_{\ep^1 \x \ga^1}) \otimes \Q$
is isomorphic to $\Omega_n \otimes \Q$.
Hence, we have the following theorem.

\begin{thm}
The rank of the simple fold bordism group $Bor_{\s}(n)$ is equal to
$\rank{\Omega_n} + \sum_{q = 0}^{n}  \rank{\Omega_q}$.\qed
\end{thm}

%

We also have that 
the direct summand $B\imm^{\ep^1}(n-1,1)$ of the simple fold bordism group $Bor_{\s}(n)$
is mapped by $\phi_n^{Bor}$ isomorphically onto a direct summand $B\imm^{\ep^1}(n-1,1)$ of 
the fold bordism group $Bor_{\f}(n)$.
Hence, we obtain the following.

\begin{cor}
The homomorphism $\phi_n^{Bor} \otimes \Q \co Bor_{\s}(n) \otimes \Q \to
Bor_{\f}(n) \otimes \Q$ is injective.
\end{cor}

\section{Appendix: Bundle structure of simple fold maps}\label{bundlestr}


Let $\si \co s \to (-1,1)$ be a representative of the ``figure eight'' singular fiber as in Figure~\ref{simplesing}. 
Let $g_1$ be an auto-diffeomorphism of $s$ and $g_2$ 
be an auto-diffeomorphism of $(-1,1)$ such that
$g_2 \circ \si = \si \circ g_1$.
We call the pair $(g_1, g_2)$ an {\it automorphism} of the representative $\si$.
The {\it automorphism group} of $\si$ consists of this kind of pairs $(g_1, g_2)$.
Let $\AUT(\si)$ denote the automorphism group of $\si$. 
Hence $\AUT(\si)$ is a subgroup of the topological group $\DIFF(s) \x \DIFF((-1,1))$.

The surface $s$ is orientable, and
let $\AUT^O(\si)$ denote the group of automorphisms $(g_1, g_2)$ in $\AUT(\si)$
such that the diffeomorphisms $g_1$ and $g_2$ both preserve or both reverse the orientations
of the manifolds $s$ and $(-1,1)$.

Let $\chi \co M^{n+1} \to N^n$ be an oriented fold map, which is a family of the Morse function $\si$, i.e., 
we have the commutative diagram
\begin{center}
\begin{graph}(6,2)
\graphlinecolour{1}\grapharrowtype{2}
\textnode {A}(1,1.5){$M^{n+1}$}
\textnode {B}(3, 0){$S_{\chi}$}
\textnode {C}(5, 1.5){$N^n$}
\diredge {A}{B}[\graphlinecolour{0}]
\diredge {C}{B}[\graphlinecolour{0}]
\diredge {A}{C}[\graphlinecolour{0}]
\freetext (1.2, 0.6){$\pi_{M}$}
\freetext (3,1.8){$\chi$}
\freetext (4.6,0.6){$\pi_{N}$}
\end{graph}
\end{center}
where $S_\chi$ denotes the indefinite fold singular set of the fold map $\chi$, and the diagram 
gives us a ``bundle'' denoted by $\csi_\si$, i.e., the ``total space'' of $\csi_\si$ is the fiberwise map
$\chi$ between the total spaces of the bundles 
$\pi_{M} \co M^{n+1} \to S_{\chi}$ with fiber $s$
and $\pi_{N} \co N^n \to S_{\chi}$ with fiber $(-1,1)$, 
the ``base space'' of $\csi_\si$ is $S_{\chi}$,
and the ``fiber'' of $\csi_\si$ 
is right-left equivalent to the Morse function $\si$.

Then the following theorem can be proved by an argument similar to that in \cite{Szucs3}.

\begin{thm}\label{szingnyalloktriv}
The bundle $\csi_\si$ is a locally trivial bundle over 
$S_{\chi}$ with the Morse function $\si$ as fiber,
and with structure group $\AUT^O(\si)$. \qed
\end{thm}

Moreover, using \cite{Jan, Wa} and giving an appropriate Riemannian metrics on $M^{n+1}$, we can prove the following.

\begin{thm}\label{strcsopveges}
The structure group of the bundle $\csi_\si$ can be reduced to the group
$\Z_2$,
where the non-trivial element of $\Z_2$ 
acts identically on $(-1,1)$ and acts on the source $s$ and on the fiber
``${\infty}$''
as a 180 degree rotation. \qed
\end{thm}

\end{document}